\documentclass[12pt]{amsart}
\usepackage{mathrsfs}
\usepackage{amsmath}
\usepackage{latexsym}
\usepackage{amssymb}
\usepackage{amsfonts}
\usepackage{multirow}
\usepackage{float}
\usepackage{enumerate}
\include{PDF}
\usepackage{graphics}

\def\l{\langle} \def\r{\rangle} 
 
 \def\ZZ{\mathbb Z}

\def\mod{{\sf mod~}}

\def\D{{\rm D}} \def\Q{{\rm Q}}

\def\Z{{\bf Z}} 

\def\Ome{{\it \Omega}}

 \def\GL{{\rm GL}}\def\PG{{\rm PG}}

\def\A{{\rm A}}\def\Sym{{\rm Sym}}

\def\PSL{{\rm PSL}}\def\PGL{{\rm PGL}}
\def\GL{{\rm GL}} \def\SL{{\rm SL}}
\def\AGL{{\rm AGL}}

  \def\D{{\rm D}}

  \def\Sym{{\rm Sym}}
  
\def\D{\mathrm{D}} \def\Q{\mathrm{Q}} \def\A{\mathrm{A}}       
 \def\PG{\mathrm{PG}}
\def\GL{\mathrm{GL}}   \def\SL{\mathrm{SL}}     \def\PGL{\mathrm{PGL}} \def\PSL{\mathrm{PSL}}         \def\AGL{\mathrm{AGL}}

\def\o{\omega}
\newcommand{\bb}{\mathbb}
\newcommand{\mf}{\mathfrak}

\newtheorem{theorem}{Theorem}[section]
\newtheorem{definition}{Definition}[section]
\newtheorem{remark}{Remark}[theorem]

\newtheorem{lemma}[theorem]{Lemma}
\newtheorem{corollary}[theorem]{Corollary}
\newtheorem{example}[theorem]{Example}

\newtheorem{problem}[theorem]{Problem}
\newtheorem{construction}[theorem]{Construction}
\newtheorem*{theorem*}{Theorem}

\def\qed{{\hfill$\Box$\bigskip}
\medbreak}

\begin{document}

\title{Erd\H{o}s-Ko-Rado problems for permutation groups}
\thanks{This work was partially supported  by NSFC grants 11771200, 11931005 and 61771019}

\author{Cai Heng Li}
\address{SUSTech International Center for Mathematics\\
Department of Mathematics\\
Southern University of Science and Technology\\
Shenzhen, Guangdong 518055\\
P. R. China}
\email{lich@sustech.edu.cn}

\author{Shu Jiao Song}
\address{School of Mathematics\\
Yantai University\\
Yantai, Shandong}
\email{shujiao.song@hotmail.com}

\author{Venkata Raghu Tej Pantangi}
\address{Department of Mathematics\\
Southern University of Science and Technology\\
Shenzhen, Guangdong 518055\\
P. R. China}
\email{pantangi@sustech.edu.cn}

\begin{abstract}
In this paper, we study intersecting sets in primitive and quasiprimitive permutation groups. Let $G\leqslant\Sym(\Ome)$ be a transitive permutation group, and $S$ an intersecting set. Previous results show that if $G$ is either 2-transitive or a Frobenius group, then $|S|\leq |G_{\o}|$ (for some $\o \in \Ome$). Furthermore, for some 2-transitive groups, $|S|=|G_{\o}|$ if and only if $S$ is a coset of a stabilizer.
In this paper, we prove that these statements are far from the truth for general transitive groups. In particular, we show that in the case of primitive groups, there is even no absolute constant $c$ such that $|S|\leqslant c|G_\o|$. In the case $G$ is a primitive permutation group isomorphic to $\PSL(2,p)$, we characterize the subgroups of $G$ which are intersecting sets. We also show that if $G \leqslant \Sym(\Ome)$ is a permutation group of prime power degree, then for any intersecting set $S$, we have $|S|\leq |G_{\o}|$ (for some $\o \in \Ome$). This proves a part of a conjecture in \cite{MRS}.
\end{abstract}

\date\today

\maketitle

\section{Introduction}
Let $G$ be a finite permutation group on $\Ome$. 
A pair of permutations $x,y \in G$ are said to be {\it intersect} if the intersection $\{(\o,\o^{x})| \ \o \in \Ome\}\cap \{(\o,\o^{y})| \ \o \in \Ome\}$ is non-empty, in other words, there exists a point $\alpha \in \Ome$ such that $\alpha^{xy^{-1}}=\alpha$.
 A subset $S$ of $G$ is said to be an {\it intersecting set} if any pair of permutations in $S$ intersect. An intersecting set is said to be {\it maximum} if it is an intersecting set of maximum possible size. 
 
Cosets of point stabilizers are some obvious examples of intersecting subsets. Thus the size of a maximum intersecting set is at least $|G_{\o}|$ (for some $\o \in \Ome$).
It is now natural to consider the following questions:
\begin{enumerate}[(A)]
\item Is the size of every intersecting set bounded above by the order of a point stabilizer?
\item Is every maximum intersecting set a coset of a point stabilizer?
\end{enumerate} 

In 1977, Deza and Frankl \cite{DF1977} showed that, for symmetric group of degree $n$, the answer to question (A) is positive, that is, the size of any intersecting set is at most $(n-1)!$. They conjectured that every maximum intersecting set is a coset of a point stabilizer. This conjecture was proved by Cameron-Ku (\cite{CK2003}) and Larose-Malvernuto (\cite{LM2004}) independently. Later in \cite{GM2009}, Godsil and Meagher gave another proof.

Deza-Frankl's work was motivated by a classical result of Erd\H{o}s-Ko-Rado in extremal set theory, which states that, for a collection $\mathfrak{S}$ of $k$-subsets of a set $\Ome$ of size $n$ with $k<n/2$, if any two members of $\mathfrak{S}$ intersect, then
\[|\mathfrak{S}|\leqslant {n-1\choose k-1},\] and equality holds if and only if $\mathfrak{S}$ consists of all $k$-subsets of $\Ome$ which contains a common element. The above mentioned questions (A) and (B), and the classical Erd\H{o}s-Ko-Rado theorem lead to the following definition.  

\begin{definition}
{\rm
A permutation group $G$ on $\Ome$ is said to have the {\it EKR property} if its point stabilizers are maximum intersecting sets. The group is further said to have the {\it strict-EKR property} if cosets of point stabilizers are the only maximum intersecting sets in $G$.
}
\end{definition}

A few classes of permutation groups satisfying the strict-EKR property are as follows: (a)Alternating groups of degree $n$ (c.f \cite{AM2014}); (b)the groups $\PGL(2,q)$ (c.f \cite{MS2011}) and  $\PSL(2,q)$ (c.f \cite{LPSX2018}) acting on $1$-spaces . It was shown in \cite{MST2016} that all $2$-transitive permutation groups have the EKR property. Not all $2$-transitive permutation groups have the strict-EKR property, for example the affine group $\AGL(1,p^{d})=\bb{Z}_{p}^{d}{:}\bb{Z}_{p^{d}-1}$ is $2$-transitive but does not have the strict-EKR property unless $(p,d)=(3,1)$. A recent result \cite{MSi2019}, shows that the characteristic vector of any maximum intersecting set of $2$-transitive groups is a linear combination of characteristic vectors of cosets of point stabilizers. Groups with this property are said to have the {\it EKR-Module} property. Theorem 3.6 of \cite{AM2015} states that all Frobenius groups have the EKR property and also that generalized dihedral groups are the only Frobenius groups satisfying the strict-EKR property.

In this paper we show that general transitive groups are ``far from satisfying the EKR property''. As means of measuring how ``far'' a group is from satisfying the EKR property, we define the following quantity.

\begin{definition}\label{ratio}
{\rm
Given a transitive permutation group $G$ on $\Ome$, by $\rho(G,\Ome)$, we denote the ratio $|S|/|G_{\o}|$, where $S$ is a maximum intersecting set and $\o \in \Ome$.
}
\end{definition}

 Since $G_{\o}$ is an intersecting set, we have $\rho(G,\Ome)\geqslant 1$. A group $G$ satisfies the EKR property if and only if $\rho(G,\Ome)=1$. The main result of  \cite{MST2016} states that all $2$-transitive groups satisfy the EKR property and so $\{\rho(G):\ G\ \text{is $2$-transitive}\}=\{1\}$. We show that $\rho(G,\Ome)$ can be arbitrarily large for general permutation groups.

\begin{theorem}\label{no-bound}
Given $M>0$, there is a transitive permutation group $G \leqslant \Sym(\Ome)$ and an intersection set $S \subseteq G$ such that $|S| >M |G_{\o}|$, where $G_{\o}$ is the stabilizer of a point $\o \in \Ome$.
\end{theorem}

This is proved by considering the groups in Example \ref{AGL-ex}. It is now natural to be curious about the upper bound of $\rho(G,\Ome)$. In Construction \ref{nobo}, we construct an infinite family of permutation groups such that for any $(G,\Ome)$ in this family, we have $\rho(G,\Ome)\leq \sqrt{|\Ome|}$. Moreover, the bound of $\sqrt{|\Ome|}$ is tight for this family.  
In a previous version of this article (\cite[Conjecture 1.2]{LSP}), we conjectured that $\sqrt{|\Ome|}$ would be the best possible bound for all transitive permutation groups. However a recent result \cite{MRS} found counterexamples to this conjecture. One of the main results of the paper is:

\begin{theorem*}[Razafimahatratra-Meagher-Spiga]
If $G\leqslant \Sym(\Ome)$ is a transitive permutation group with $|\Ome|>3$, then $\rho(G,\Ome)\leq |\Ome|/3$.
\end{theorem*}
Using an exhaustive computer search on the database of permutation groups of degree at most $48$, the authors of \cite{MRS} found four examples of permutations group with $\rho(G,\Ome)=|\Ome|/3$. Moreover in three of these examples, $\rho(G,\Ome)>\sqrt{|\Ome|}$. In \ref{lisg}, we describe these examples and two other groups with $\rho(G,\Ome)>\sqrt{|\Ome|}$. We make use of wreath products and the five examples in \ref{lisg} to produce infinite families of  permutation groups with $\rho(G,\Ome) \geqslant \sqrt{|\Ome|}$. Since wreath products were essential to our constructions of infinite families with ``large'' intersecting subgroups, we divide all permutation groups into two categories.
A permutation group $G\leqslant \Sym(\Ome)$ is said to have a {\it product decomposition}, if there exists permutation groups $H \leqslant \Sym(\Delta)$ and $K \leqslant \Sym(\Theta)$ such that $G$ is permutation equivalent  to $H{\wr} K \leqslant \Sym(\Delta^\Theta)$. The groups mentioned in \ref{lisg} do not have a product decomposition. So far, in both this article and \cite{MRS}, there are only finitely many examples of permutation groups $G\leqslant \Sym(\Ome)$ which do not admit a product decomposition and satisfy $\rho(G,\Ome)> \sqrt{|\Ome|}$.      

\begin{theorem}\label{sqo}
Given $M>0$ and $\epsilon \in (0,1)$, there exists a permutation group $G\leqslant \Sym(\Ome)$ which does not admit a product decomposition, such that $\rho(G,\Ome)>M$ and $(1-\epsilon)\sqrt{|\Ome|} <\rho(G)<\sqrt{|\Ome|}$.
\end{theorem}        

\begin{theorem}\label{3bou}
There are infinitely many permutation groups $G\leqslant \Sym(\Ome)$ admitting a product decomposition and satisfying $\rho(G,\Ome)>\sqrt{|\Ome|}$.
\end{theorem}
We now reformulate our Conjecture 1.2 of \cite{LSP} as the following question for future research.
\begin{problem}{\rm
 Are there infinitely many permutation groups $G \leqslant \Sym(\Ome)$ which do not admit a product decomposition and satisfy $\rho(G,\Ome)>\sqrt{|\Ome|}$?
}
\end{problem}

There are results describing the maximum intersecting sets of Frobenius groups (c.f \cite[Theorem 3.6, 3.7]{AM2015}) and of $2$-transitive groups ((c.f \cite{MST2016, MSi2019})). A natural next step is to carry similar investigations for primitive permutation groups. A permutation group is called {\it quasiprimitive} if each of its non-trivial normal subgroup is transitive (All primitive permutation groups are quasiprimitive but the converse is not true.). In this paper, we will see that although many quasiprimitive groups have the EKR property, infinitely many of them do not have the EKR property and  moreover have ``large'' intersecting sets. The O'Nan-Scott-Praeger classification theorem for quasiprimitive groups (c.f \cite{Pra1993}) shows that there are $8$ types of quasiprimitive groups. For our purposes, we divide them into three types: (a) those with subnormal regular subgroups, (b) almost simple, and (c) quasiprimitive groups of product action type. As quasiprimitve groups of the first type have a regular normal subgroup, an application of Lemma 14.6.1 of \cite{GMbook} implies that they satisfy the EKR property. However, Constructions \ref{ASC} and \ref{PAC}, show that $\rho(G, \Ome)$ can be arbitrarily large for quasiprimitive groups of the other two types. We note that both the groups in both these constructions are in fact primitive. So even primitive groups are quite ``far'' from satisfying the EKR property.

\begin{theorem}\label{qp-EKR}
Let $G$ a quasiprimitive permutation group on a set $\Ome$, and let $\rho(G,\Ome)$ be as defined in $\ref{ratio}$. Then either (i) $\rho(G,\Ome)=1$ and $G$ has the EKR property, or (ii) $G$ is an almost simple group or a group of product action type. Moreover for primitive groups in case~{(ii)}, $\rho(G,\Ome)$ can be arbitrarily large. In other words, given $M>0$, there is an almost simple primitive group $A \leqslant \Sym(\Gamma)$ and a primitive group $P \leqslant \Sym(\Delta)$ of product action type, such that $\rho(A, \Gamma)>M$ and $\rho(P, \Delta)>M$.
\end{theorem}

Two problems now naturally arise, corresponding to the two parts of Theorem~\ref{qp-EKR}.

\begin{problem}\label{problem-1}
{\rm
\begin{enumerate}[{\rm(a)}]
\item Determine which quasiprimitive permutation groups with a subnormal regular subgroup have the strict-EKR property.
\item Determine which quasiprimitive almost simple groups have the EKR property.
\end{enumerate}
}
\end{problem}

The two problems in Problem~\ref{problem-1} seem difficult. For a possible approach, we propose to study a weaker property defined below.
{\rm
\begin{definition}\label{weak-EKR}
{\rm
Let $G$ be a transitive permutation group on a set $\Ome$, and $\o\in\Ome$.
\begin{enumerate}[{\rm(a)}]
\item A subgroup $S$ is called an {\it intersecting subgroup} if it is an intersecting set in $G$;

\item $G$ is said to have the {\it weak-EKR property} if each intersecting subgroup has size at most $|G_\o|$;
\item $G$ is said to have the {\it strict-weak-EKR property} if it has the weak-EKR property and all $|G_{\o}|$ sized intersecting subgroups are conjugate to $G_\o$.
\end{enumerate}
}
\end{definition}
}
Our proof of Theorem~\ref{qp-EKR} is based on constructing groups with large intersecting subgroups. To do so, we will use characterizations of intersecting subgroups given in Lemma \ref{Int-subgp}. In view of Theorem \ref{qp-EKR} and Lemma \ref{Int-subgp}, it is natural to ask the following problem which is a weaker version of Problem \ref{problem-1}.

\begin{problem}\label{AS-EKR}
{\rm
\begin{enumerate}[{\rm(a)}]
\item Classify all almost simple primitive permutation groups $G \leqslant \Sym(\Omega)$ such that $|H| \leqslant |G|/|\Ome|$, for any subgroup $H<G$ with $H \subset \bigcup\limits_{\o \in \Ome}G_{\o}$.
\item Classify all almost simple primitive permutation groups $G \leqslant \Sym(\Omega)$ such that any subgroup $H<G$ with $H \subset \bigcup\limits_{\o \in \Ome}G_{\o}$ is a point stabilizer.
\end{enumerate}
}
\end{problem}
As a starting point, we solve this problem for primitive permutation groups isomorphic to $\PSL(2,p)$. We prove the following theorem in Section~\ref{psl}. Statement~(iii) of the Theorem is the main result of \cite{LPSX2018}.
\begin{theorem}\label{PSLEKR}
Let $p\geqslant 5$ be a prime, $G=\PSL(2,p)$ be a primitive permutation group on $\Ome$ and $\o \in \Ome$. Then the following are true.
\begin{enumerate}[{\rm(i)}]
\item If $p\neq 11,\ 13,\ 23,\ 29,\ 31,\ 59,$ or $61$, then $G$ has the strict-weak-EKR property.
\item If $p\neq 29$ or $31$, then $G$ has the weak-EKR property.
\item If $G_{\o} \cong \bb{Z}_{p}{:}\bb{Z}_{p-1}$, then $G$ has the strict-EKR property.
\item If $p\equiv 3\pmod{4}$ and $G_{\o}\cong \D_{p+1}$, then $G$ has the EKR property.
\item If $p=29$ or $31$, then $G$ has the strict-weak-EKR property unless $G_{\o}\cong \D_{30}$.
\item If $p= 13$ or $61$, then $G$ has the strict-weak-EKR property unless $G_{\o}\cong \D_{p-1}$.
\item If $p= 11$, $23$ or $59$, then $G$ has the strict-weak-EKR property unless $G_{\o}\cong \D_{p+1}$.
\end{enumerate}
\end{theorem}
\begin{remark}
{\rm We note that in statements (i),(ii) of the above theorem are true for all primitive representations of $\PSL(2,p)$. On the other hand, in the other statements, there are additional restrictions of the structure of the stabilizer $G_{\o}$.        
}
\end{remark}

We end this paper by proving part (2) of Conjecture 6.6 of \cite{MRS}.
It is easy to see that any subgroup of order $p$ in a transitive permutation group of prime degree is a regular subgroup. Thus by Lemma~\ref{regekr}, any prime degree permutation group satisfies the EKR property. Based on this and some computational evidence, the authors of \cite{MRS} conjectured (Conjecture 6.6 (2) of \cite{MRS}) that
all permutation groups of prime power degree satisfy the EKR property. In Section~\ref{pp}, we show that this conjecture is true.
\begin{theorem}\label{ps}
All transitive permutation groups of prime power degree satisfy the EKR property. 
\end{theorem}

\section{Proof of Theorem~\ref{no-bound} and Theorem~\ref{sqo}.}\label{prel}
In this section we construct permutation groups with large intersecting sets, and thereby prove Theorem~\ref{no-bound}. We start by making the following elementary observation.

\begin{lemma}\label{Int-subgp}
Let $G$ be a transitive permutation group on $\Ome$, and fix a point $\o\in\Ome$.
Then, for each subgroup $S<G$, the following statements are equivalent:
\begin{enumerate}[{\rm(a)}]
\item $S$ is an intersecting subgroup,
\item every element of $S$ fixes some point of $\Ome$,
\item each element of $S$ is conjugate to an element of $G_\o$.
\end{enumerate}
\end{lemma}
\begin{proof}
First we show that (a) implies (b). As $S$ is an intersecting set, given $h\in S$, the ratio $he^{-1}=h$ must fix some point of $\Ome$.

Let $h\in G$ fix a point $\alpha \in \Ome$, that is $h \in G_{\alpha}$. As $G$ acts transitively, the stabilizers $G_{\alpha}$ and $G_{\o}$ are conjugate, and thus $h$ is conjugate to some element of $G_{\o}$. This observation shows that (b) implies (c).

Now let $h,x\in G$ be such that $x^{-1}hx\in G_{\o}$. It follows that $h$ fixes $\o^{x}$. This observation shows that (c) implies (a).    
\end{proof}
 
It would be worth to point out that group theoretic properties (solubility, nilpotence etc) of the stabilizer subgroup may not hold for intersecting subgroups. We give the following examples to support this statement.

\begin{example}
{\rm
\begin{enumerate}

\item
Assume $p$ is a prime with $p\equiv 1$ $(\mod 30)$, say $p=31$.
Consider $G=\PSL(2,p)$ and a subgroup $H$ which is isomorphic to the dihedral group $\D_{p-1}$.
In this case, $G$ has a subgroup $S$ which isomorphic to the alternating group $\A_5$.
It is known that all subgroups of $G$ of order 2, 3 or 5 are conjugate.
As $|H|$ is divisible by 30, each element of $S=\A_5$ is conjugate to an element of $H$.
Thus by Lemma \ref{Int-subgp}, we see that $S$ is an intersecting subgroup of $G$, with respect to its action on $[G:H]$.
Note that {\it $H\cong \D_{p-1}$ is solvable, and $S\cong \A_5$ is non-solvable.}

\item Let $G=\PSL(2,2^f)$ where $f=2r$ with $r\geqslant2$.
Consider a subgroup $H$ isomorphic to $\PSL(2,2^2)$.
It is known that all involutions of $G$ are conjugate.
It follows that a Sylow 2-subgroup $S\cong\ZZ_2^f$ is an intersecting set of $G$, with respect to its action on $[G:H]$.
Note that {\it $H\cong \PSL(2,2^2)\cong\A_5$ is non-solvable, and $S$ is an elementary abelian group.}

\end{enumerate}}
\end{example}

We now apply Lemma \ref{Int-subgp} to construct large intersecting sets and thus prove Theorem \ref{no-bound}.
\begin{example}
\label{AGL-ex}
{\rm
Let $p$ be an odd prime, $V=\ZZ_p^{d}$, and $G=\AGL(d,p)=V{:}\GL(d,p)$ be the affine group.
Pick an involution $x$ of the centre $\Z(\GL(d,p))$, and
another involution $y$ in $(V{:}\l x\r)\setminus\l x\r$.
Then $G_\o:=\l x,y\r\cong \D_{2p}$. We now consider the action of $G$ on $[G:G_\o]$.

Since $\GL(d,p)$ acts transitively on the set of non-zero vectors, we see that every element of $S=V{:}\l x\r$ is conjugate to some element of $G_{\o}$. Thus
by Lemma~\ref{Int-subgp}, we see that $S$ is an intersecting set of $G$. We note that
\[|S|=2p^d,\ \text{and} |G_\o|=2p.\]
So $|S|$ can be arbitrarily larger than $|G_\o|$, dependent on $d$.}
\end{example}

We now recall a result from \cite{AM2015} that will be used in proofs of our main results. A subset $C$ of a permutation group $G \leqslant \Sym(\Ome)$ is said to be {\it sharply transitive} if for any $(\alpha,\beta) \in \Ome \times \Ome$ there is exactly one $c \in C$ such that $\alpha^{c}=\beta$. We note that a regular subgroup is a sharply transitive set. The following result which is equivalent to Corollary 2.2 of \cite{AM2015} shows that permutation groups containing sharply transitive sets satisfy the EKR property. The authors of \cite{AM2015} prove it using some Linear algebra and graph theory. In Section~\ref{FR}, we present an alternate proof for the interested reader. In Section~\ref{pp}, we use this result to prove Theorem~\ref{ps}.  

\begin{lemma}\label{regekr}
Let $G$ be a finite permutation group on $\Omega$ and $\o \in \Omega$. If $G$ contains a sharply transitive subset $C$, then the following are true:
\begin{enumerate}[{\rm(a)}]
\item If $S \subset G$ is an intersecting set, then there exists a set $\{H_{c}|\ c \in K\}$ of pairwise disjoint subsets of $H$ (that is $H_{c}\subset H$ and $H_{c}\cap H_{d} =\emptyset$ for all $c, d \in K$ with $c\neq d$) such that $S=\bigcup\limits_{c \in K} H_{c}c$; and
\item $G$ has the EKR property
\end{enumerate}
\end{lemma}
\begin{proof}
A proof is given in Section \ref{pre}.
\end{proof}

The following construction which is motivated by the Example \ref{AGL-ex} will be used to prove Theorem \ref{no-bound}. 

\begin{construction}\label{nobo}{\rm
Let $p$ be a prime, $q=p^{d}$. Let $E=(\bb{F}_{q},+)$ and $F=(\bb{F}_{q^{2}},+)$ be  additive groups of the finite fields of orders $p^{d}$ and $p^{2d}$ respectively. Let $E^{\times}$ and $F^{\times}$ denote the multiplicative groups of the finite fields of orders $p^{d}$ and $p^{2d}$ respectively. We consider the action of the affine group $G=\AGL(1,p^{2d})=F:F^{\times}$ on the coset space $\Omega=[G:H]$ of the subgroup $H=E{:}E^{\times}$. This action makes $G$ a transitive permutation group on $\Omega$. Pick the subgroup $S=F:E^{\times}$.}
\end{construction} 

\begin{lemma}
Let $G$, $H$, $\Omega$, and $S$ be as defined in Construction~\ref{nobo}. Then the following are true.
\begin{enumerate}[{\rm(a)}]
\item $S$ is a maximum intersecting set,
\item $|S| < |\sqrt{\Ome}||H|$, and
\item $\lim\limits_{q\to \infty}\dfrac{|S|}{\sqrt{|\Omega|}|H|}=1$.
\end{enumerate}
\end{lemma}
\begin{proof}

Let $x \in F^{\times}$ and $y \in E^{\times}$, and consider the element $(x,y)\in S$. We have $(0,x^{-1})(x,y)(0,x)=(1,y) \in E{:}E^{\times}=H.$ Thus every element of $S$ is conjugate to some element of $H$. So by Lemma~\ref{Int-subgp}, we see that $S$ is an intersecting subgroup. Since $H\leqslant S$, every element of $H$ is conjugate to an element of $S$. Thus $g \in G$ is conjugate to an element $H$ if and only if it is conjugate to an element of $S$. Now $g\in G$ fixes a point in the coset space $[G:S]$  if and only if $g$ is conjugate to an element of $S$. Therefore a subset $X$ is an intersecting set with respect to the action of $G$ on $\Omega=[G:H]$ if and only if it is an intersecting set with respect the action of $G$ on $[G:S]$.

Now the $q+1$ element cyclic subgroup $K$ of $F^{\times}$ is a regular subgroup with respect to the action of $G$ on $[G:S]$. By Lemma~\ref{regekr}, every permutation group with a regular subgroup (which is a sharply transitive set) must satisfy the EKR property. Therefore $S$ is a maximum intersecting set for the action of $G$ on $[G:S]$ and hence also for the action of $G$ on $\Omega$.

The size of $\Ome$ is equal to\[|\Ome|=\frac{|G|}{|H|}=\frac{q^{2}(q^{2}-1)}{q(q-1)}=q(q+1).\]
On the other hand we have, \[\frac{|S|}{|H|}=\frac{q^{2}(q-1)}{q(q-1)}=q= \sqrt{q^{2}}<\sqrt{q(q+1)}=\sqrt{|\Omega|}.\] We finish the proof by observing that $\lim\limits_{q\to \infty}\dfrac{|S|}{\sqrt{|\Omega|}|H|}=1$. 
\end{proof}

To conclude the proof of Theorem~\ref{sqo}, we need to show that $G$ does not admit a product decomposition. The degree of $G$ as a permutation group is $q(q+1)$, which is not of the form $n^{m}$, and thus $G$ does not admit a product decomposition.
 
{\bf Proof of Theorem~\ref{sqo}} is now complete by the above Lemma.
\section{Proof of Theorem~\ref{3bou}}
We start this section by recalling a construction of permutation groups using the wreath product. Let $\ell$ be a prime. Consider the group $\bb{Z}_{\ell}$ and a set $X$. By $X^{\bb{Z}_{\ell}}$, we denote the set of $X$-valued functions on $\bb{Z}_{\ell}$. The group $\bb{Z}_{\ell}$ has a natural action on $X^{\bb{Z}_{\ell}}$ given by $(m\cdot \pi)(j)=\pi(j-m)$, for all $m \in \bb{Z}_{\ell}$, $\pi \in X^{\bb{Z}_{\ell}}$ and $j \in \bb{Z}_{\ell}$.

Consider a permutation group $G \leqslant \Sym(\Ome)$. Then the set $G^{\bb{Z}_{\ell}}$ can be considered as a group acting on $\Ome^{\bb{Z}_{\ell}}$ via the formula $(f\pi)(j)=f(j)\pi(j)$ for all $f \in G^{\bb{Z}_{\ell}} $, $\pi \in X^{\bb{Z}_{\ell}}$ and $j \in \bb{Z}_{\ell}$. We note that $\bb{Z}_{\ell}$ normalizes $G^{\bb{Z}_{\ell}}$ in $\mathrm{Aut}(X^{\bb{Z}_{\ell}})$, with $mfm^{-1}=m\cdot f$, and so we have $G^{\bb{Z}_{\ell}}\ : \bb{Z}_{\ell} \leqslant \Sym(\Ome^{\Gamma})$. The element $(f;m)\in G^{\bb{Z}_{\ell}}\ : \bb{Z}_{\ell}$ acts on $\pi \in \Ome^{\Gamma}$ via $((f;m)\pi)(i)=f(i-m)\pi(i-m)$. We use $G {\wr} \bb{Z}_{\ell}\leqslant \Sym(\Ome^{\Gamma})$ to denote this permutation group. The following Lemma will help us understand the relationship between $\rho(G, \Ome)$ and $\rho(G\ {\wr} \bb{Z}_{\ell}, \Ome^{\bb{Z}_{\ell}})$, in the cases where we have a subgroup which is also a maximum intersecting set.

\begin{lemma}\label{paty}
If $S$ is an intersecting subgroup of $G \leqslant \Sym(\Ome)$, then $\mathfrak{S}=S\ {\wr}\bb{Z}_{\ell}$ is an intersecting subgroup of $G {\wr} \bb{Z}_{\ell} \leqslant \Sym(\Ome^{\bb{Z}_{\ell}})$.
\end{lemma}
\begin{proof}
Consider $f \in S^{\bb{Z}^\ell}$ and $m\neq 0 \in \bb{Z}_{\ell}$. We now show that $(f;m)\in S{\wr}\bb{Z}_{\ell}$ fixes a point of $\Ome^{\bb{Z}_{\ell}}$. Since every element of $S$ fixes a point in $\Ome$ (by Lemma~\ref{Int-subgp}), we may choose a point $x_{-m}\in \Ome$ such that it is fixed by $\left(\prod \limits_{j=0}^{\ell-1}f({(\ell-j-2)m})\right)$. For $0 \leqslant i \leqslant \ell-2$, we set \[x_{im}:=\left(\prod\limits_{j=i+1}^{\ell-1}f(jm)\right) x_{-m}.\] We thus have $f((i-1)m)x_{(i-1)m}=x_{im}$. Note the $m$ generates $\bb{Z}_{\ell}$ and thus \[\{im|\ -1\leqslant i \leqslant \ell-2\}=\bb{Z}_{\ell}.\] Let $\pi \in \Ome^{\bb{Z}_{\ell}}$ be such that $\pi(im)=x_{im}$ for all $-1\leq i \leq \ell-2$, then $(f;m)\pi(im)=f((i-1)m)\pi((i-1)m)=\pi(im)$, and therefore $(f;m)$ fixes $\pi$. Thus by Lemma \ref{Int-subgp}, $\mathfrak{S}$ is an intersecting subgroup.
\end{proof}
 
We are now ready to prove Theorem~\ref{3bou}. Let $P\leqslant \Sym(\Delta)$ be a permutation group with a maximum intersecting set $S$, which is also a subgroup. Then by Lemma~\ref{paty}, $\mathfrak{S}=S {\wr} \bb{Z}_{\ell} \leqslant \Sym(\Ome^{\bb{Z}_{\ell}})$ is an intersecting subgroup of $P \wr \bb{Z}_{\ell}\leqslant \Sym(\Ome^{\bb{Z}_{\ell}})$ (here $\ell$ is a prime). Given $\delta \in \Delta$, the group $P_{\delta}^{\bb{Z}_{\ell}} {:} \bb{Z}_{\ell}$ is a point stabilizer in the permutation group $P \wr \bb{Z}_{\ell}\leqslant \Sym(\Ome^{\bb{Z}_{\ell}})$. Thus we have \[\rho(P \wr \bb{Z}_{\ell}, \Ome^{\bb{Z}_{\ell}})\geq \dfrac{|\mathfrak{S}|}{|P_{\delta}^{\bb{Z}_{\ell}} {:} \bb{Z}_{\ell}|}= \left(\dfrac{|S|}{|P_{\delta}|}\right)^{\ell}.\] So if we find a permutation group $P \leqslant \Sym(\Delta)$ which has an intersecting subgroup $S <P$ satisfying $\dfrac{|S|}{|P_{\delta}|}> \sqrt{|\Delta|}$, then we would have proved Theorem~\ref{3bou}. Theorem 5.1 of \cite{MRS} provides three such examples. The authors found these using magma. In Table~\ref{larget}, we provide examples of a groups $G$ acting on the coset space $\Ome$ of a subgroup $H$. In all these examples, there is an intersecting subgroup $K$ with $|K|/|H|>\sqrt{|\Ome|}$. Therefore for all the groups in the table, we have\[\rho(G,\Ome) \geq |K|/|H| > \sqrt{|\Ome|}\]
The groups in rows 2, 3, and 6, were found by the authors of \cite{MRS}. 
In the following subsection, we will explain these examples in detail.
\subsection{Groups with large intersecting subgroups.}\label{lisg}
\footnotesize
\begin{table}[H]
\begin{tabular}{|c |c |c |c |c |}
\hline
$G$ & $H$ & $|\Ome|=|G|/|H|$ & $K$ & $|K|/|H|$ \\
\hline
$\bb{F}_{5}^{2}\ :\SL(2,3)$ & $\bb{F}_{5}\ :\bb{Z}_{4}$ & $30$ & $\bb{F}_{5}^{2}\ :Q_{8}$& $10$\\
\hline
$\bb{F}_{5}^{2}\ :(\SL(2,3)\ : \bb{Z}_{2})$ & $\bb{F}_{5}\ :(\bb{Z}_{4}\ : \bb{Z}_{2})$ & $30$ & $\bb{F}_{5}^{2}\ :(Q_{8}\ : \bb{Z}_{2})$ & $10$ \\
\hline
$\bb{F}_{29}^{2}\ :(\SL(2,5)\times \bb{Z}_{7})$ & $\bb{F}_{29}\ :(\bb{Z}_{4}\times \bb{Z}_{7})$ & $29*30$ & $\bb{F}_{29}^{2}\ :(Q_{8}\times \bb{Z}_{7})$ & $29*2$ \\
\hline
$\bb{F}_{29}^{2}\ :(\SL(2,5)\circ \bb{Z}_{28})$ & $\bb{F}_{29}\ :(\bb{Z}_{4}\circ\bb{Z}_{28})$ & $29*30$ & $\bb{F}_{29}^{2}\ :(Q_{8}\circ \bb{Z}_{28})$ & $29*2$\\ 
\hline
$\bb{F}_{3}^{3}:\A_4$ & $\bb{F}_{3}^{2}: \bb{Z}_{2}$ & $18$ & $\bb{F}_{3}^{3}: \bb{Z}_{2}^{2}$ & $6$\\
\hline 
\end{tabular}
\caption{Groups with large intersecting subgroups.}
\label{larget}
\end{table}
\normalsize
We begin by recalling some general results about subgroups of $\SL(2,p)$. Let $p$ be a prime, then the following hold:\\ (a)if $p\equiv \pm 1 \pmod{10}$, then $\SL(2,p)$ has a subgroup isomorphic to $\SL(2,5)$; and\\ (b)  if $p\equiv \pm 1 \pmod{4}$, then $\SL(2,p)$ has a subgroup isomorphic to $\SL(2,3)$. Throughout this section, given $a,b \in \bb{F}_{p}^{2}$, by $(a,b)^{\dagger}$ we denote the transpose of the row vector $(a,b)$.

The Sylow-2-subgroups of $\SL(2,3)$ and $\SL(2,5)$ are isomorphic to the Quaternion group $Q_{8}$. The polynomial $x^{2}+1$ is the minimal polynomial of any order $4$ element of $\SL(2,p)$. Thus all elements of order $4$ in $\SL(2,p)$ are conjugate. So if $Q< \SL(2,p)$, with $Q \cong Q_{8}$, then every element of $Q$ is $\SL(2,p)$-conjugate to an element of $V$, where $V<Q$ is an order $4$ subgroup.

Let $p\equiv \pm 1 \pmod{4}$ with $p>3$, and let $S\cong \SL(2,3)$ be a subgroup of $\SL(2,p)$. An element $M\in \SL(2,p)$ fixes a non-zero vector of $\bb{F}_{p}^{2}$ if and only if it has $1$ as an eigenvalue. Thus $M$ fixes a non-zero vector if and only if $M$ is conjugate to an element of $\left\{\begin{pmatrix}
1 & b\\
0 & 1
\end{pmatrix}|\ b\in \bb{F}_{p} \right\}$. So if $M\neq I$ fixes a non-zerovector, then the order of $M$ is $p$. This shows that $S$ acts semi-regularly on $\bb{F}^{2}_{p}\setminus \{(0,0)^{\dagger}\}$. A similar argument shows that, if $p \equiv \pm 1 \pmod{10}$, then any subgroup $T \cong \SL(2,5)$ of $\SL(2,p)$ acts semi-regularly on $\bb{F}^{2}_{p}\setminus \{(0,0)^{\dagger}\}$. 
\begin{lemma}\label{fff}
\begin{enumerate}[(\rm{a})]
\item If $p\equiv \pm 1 \pmod{4}$, then any subgroup $S \cong \SL(2,3)$ of $\SL(2,p)$ acts semi-regularly on $\bb{F}^{2}_{p}\setminus \{(0,0)^{\dagger}\}$.
\item  If $p\equiv \pm 1 \pmod{10}$, then any subgroup $T \cong \SL(2,5)$ of $\SL(2,p)$ acts semi-regularly on $\bb{F}^{2}_{p}\setminus \{(0,0)^{\dagger}\}$.
\end{enumerate}
\end{lemma}     

\begin{example}
{\rm
(1) Let $S$ be a subgroup of $\SL(2,5)$ which is isomorphic to $\SL(2,3)$. Consider the elements $\mf{i}=\begin{pmatrix}
0 & 1\\
-1 & 0
\end{pmatrix}$ and $\mf{j}=\begin{pmatrix}
2 & 0\\
0 & 3
\end{pmatrix}$ of $\SL(2,5)$. The subgroup $Q$ generated by $\mf{i},\mf{j}$ is isomorphic to the Quaternion group. As Sylow-2-subgroups of both $\SL(2,3)$ and $\SL(2,5)$ are isomorphic to $Q$, we may choose $S$ to contain $Q$. Let $V:=\langle \mf{j} \rangle$. We now consider the action of $G:= \bb{F}_{5}^{2} {:} S$ on the coset space $\Ome$ of $H:= \bb{F}_{5} {:} V$. As $V<Q\triangleleft S$, we see that $\cup_{g\in G}H^{g} \subset K:= \bb{F}_{5}^{2} \times Q$. As any element of $Q$ is $S$-conjugate to an element of $V$, we see that any element of $K$ is $G$-conjugate to an element of  $\{(v, \pm \mf{j}), (v, \pm I)|\ v\in \bb{F}_{5}^{2}\}$. We observe that given $a,b \in \bb{F}_{5}$, we have
$$((0,a)^{\dagger},-\mf{i})((a,b)^{\dagger}, \pm\mf{j})((-a,0)^{\dagger},\mf{i})=((-b,0)^{\dagger},\ \mp\mf{j})\in H.$$
If the action of $S$ on $\bb{F}_{5}^{2}\setminus \{(0,0)^{\dagger}\}$ is transitive, then every element of the form $(v, \pm I)$ is conjugate to $((1,0)^{\dagger}, \pm I) \in H$. Thus if this is true, we have $K := \cup H^{g}$ and thus $K$ has to be a maximum intersecting set. From Lemma \ref{fff}, $S$ acts semi-regularly on $\bb{F}_{5}^{2} \setminus \{(0,0)^{\dagger}\}$. Since $|S|=5^2-1$, $S$ acts regularly, and thus transitively on $\bb{F}_{5}^{2} \setminus \{(0,0)^{\dagger}\}$. So in this case,
$$\rho(G, \Ome)= |K|/|H|= 10 =|\Ome|/3 > \sqrt{|\Ome|}.$$ This example satisfies the parameters in row 2 of Table~\ref{larget}.

(2) Let $\mf{b}=\begin{pmatrix}
1 &0\\
0 &-1
\end{pmatrix} \in \GL(2,5)$. Observing that $\mf{b}$ normalizes $S=\langle\{\mf{i},\mf{j}\} \rangle$, we note that $\langle S, \mf{b}\rangle \cong \SL(2,3) {:}\bb{Z}_{2}$, $\langle Q, \mf{b}\rangle \cong Q_{8} \ :\bb{Z}_{2}$, and $\langle V, \mf{b}\rangle \cong \bb{Z}_{4} {:}\bb{Z}_{2}$. Consider the action of $\tilde{G}:= \bb{F}_{5}^{2}{:}\langle S, \mf{b}\rangle$ on the coset space $\tilde{\Omega}$ of the subgroup $\tilde{H}:= \bb{F}_{5}{:}\langle S, \mf{b}\rangle \cong \bb{F}_{5}^{2}{:}(\bb{Z}_{4}{:}\bb{Z}_{2})$. We now show that $\tilde{K}:=\bb{F}_{5}^{2}{:}\langle Q, \mf{b}\rangle \cong \bb{F}_{5}^{2}{:}(Q_{8}{:}\bb{Z}_{2})$ is a maximum intersecting set. Observing that $\langle V, \mf{b}\rangle< \langle Q, \mf{b}\rangle\triangleleft \langle S, \mf{b}\rangle $, we can conclude that $\cup_{g \in \tilde{G}}\tilde{H}^{g} \subset \tilde{K}$. As any element of $\langle Q, \mf{b}\rangle$ is $\langle S, \mf{b}\rangle$-conjugate to an element of $\langle V, \mf{b}\rangle$, we see that any element of $\tilde{K}$ is $\tilde{G}$-conjugate to an element of $\{(v,h)|\ v\in \bb{F}_{5}^{2}\ \&\ h\in \tilde{H}\}$. As $S$ acts transitively on $\bb{F}_{5}^{2} \setminus \{(0,0)^{\dagger}\}$, we see an element of the form $(v, \pm I)$ is conjugate to $((1,0)^{\dagger}, \pm I) \in \tilde{H}$. As $\mf{jb}=\pm 2I$, every element of the form $(v, \pm 2I)$ is conjugate to $((1,0)^{\dagger}, \pm 2I) \in \tilde{H}$. Lastly, we observe that 
\begin{align*}
((0,a)^{\dagger},-\mf{i})((a,b)^{\dagger}, \pm\mf{j})((-a,0)^{\dagger},\mf{i})=((-b,0)^{\dagger},\ \mp\mf{j})\in \tilde{H}\\
((0,-b)^{\dagger},2I)((a,b)^{\dagger}, \pm \mf{b})((2b,0)^{\dagger},-2I)=((-b,0)^{\dagger},\ \mp\mf{j})\in \tilde{H},
\end{align*}
and thus conclude that every element of $\tilde{K}$ is conjugate to an element of $\tilde{H}$. Therefore $\tilde{K}=\cup_{g\in \tilde{G}}\tilde{H}^{g}$ and thus $\tilde{K}$ is a maximum intersecting set.

In this case, we have
$$\rho(\tilde{G}, \tilde{\Ome})= |\tilde{K}|/|\tilde{H}|= 10 =|\tilde{\Ome}|/3 > \sqrt{|\tilde{\Ome}|}.$$ 
This example satisfies the parameters in row 3 of Table~\ref{larget}.
}
\end{example}
\begin{example}
{\rm
(1) Let $T$ be a subgroup of $\SL(2,29)$ which is isomorphic to $\SL(2,5)$.

Let $\eta$ be a primitive $4$-th root of unity in $\bb{F}_{29}$, and let $\zeta$ be a primitive  $7$-th root.  
Consider the elements  $\mf{i}=\begin{pmatrix}
0 & 1 \\
-1 & 0
\end{pmatrix} $, $\mf{j}=\begin{pmatrix}
\eta & 0\\
0 & -\eta
\end{pmatrix}$ and $z:=\begin{pmatrix}
\zeta & 0\\
0 & \zeta
\end{pmatrix}$ of $\GL(2,29)$. The group $Q=\langle \mf{i}, \mf{j}\rangle \cong Q_8$ is a Sylow-2-subgroup of $\SL(2,29)$. We recall that Sylow-2 subgroup of $\SL(2,5)$ is also isomorphic to $Q_8$. Thus we may assume that $Q<T$. As $z$ is a central element of order $7$, we see that $X :=\langle S,z\rangle\cong \SL(2,5)\times \langle z \rangle$. Let $\{e_{1},e_{2}\}$ be the basis of elementary vectors for $\bb{F}_{29}^{2}$.
We now consider the action of $G:=\bb{F}_{29}^2 {:} X\cong \bb{F}_{29}^2 {:} (\SL(2,5) \times \bb{Z}_{7})$ on the coset space $\Ome$ of $H:=\langle e_{1}\rangle {:} \langle \mf{j},z\rangle \cong \bb{F}_{29} {:} (\bb{Z}_{4} \times \bb{Z}_{7})$. 

Consider the subgroup $K:= \bb{F}_{29}^{2} {:} Q\langle z \rangle \cong \bb{F}_{29}^{2} {:} (Q_{8} \times \bb{Z}_{7})$. We observe that $s^{-1}\mf{j}s= \mf{i} \mf{j}$ and $s\mf{j}s^{-1}=\mf{i}$, and thus any element of $K$ is conjugate to an element of \[\{(v,\pm \mf{j}z^{l}), (v, \pm z^{l})|\ v \in \bb{F}_{29}^{2}\ \text{and}\ 0\leqslant l \leqslant 6\}.\] Given $a,b \in \bb{F}_{29}$, we see that 
$$((-a/11)e_{2},-\mf{i})(ae_{1}+be_{2}, \pm\mf{j}z^{l})((a/11)e_{1},\mf{i})=(-be_{1}, \mp\mf{j}z^{l})\in H.$$  

Given $t \in X$, as $\pm z^{l}$ is a central element, we have $(0,t)(v,\pm z^{l})(0,t^{-1})=(t\cdot v,\ \pm z^{l})$. We claim that the action of $X$ on non-zero vectors of $\bb{F}_{q}^{29}$ is regular, and hence for any $v \in \bb{F}_{29}^{2}$, there exists $t\in T$ such that $t\cdot v =e_{1}$. This shows that by Lemma \ref{Int-subgp}, $K$ is an intersecting subgroups. So we have 
$$\rho(G,\Omega) \geq |K|/|H|=29\times 2 > \sqrt{|\Ome|}=\sqrt{29\times 30}.$$

We now show that the action of $X$ on $\bb{F}_{29}^{2} \setminus \{0\}$ is regular. By Lemma~\ref{fff}, $T$ acts semi-regularly. Since $z$ does not fix any non-zero vector, $X=\langle T,z \rangle$ acts semi-regularly. As $|T|=|\bb{F}_{29}^{2}|-1$, we conclude that $T$ acts regularly on the set of non-zero vectors of $\bb{F}_{29}^2$.
This example satisfies the parameters in row 4 of Table~\ref{larget}.

(2) Replacing $\zeta$ with a primitive $28$th root of unity will give us the example satisfying the parameters in row 5 of Table~\ref{larget}. 
}
\end{example}
\begin{example}{\rm 
Let $\{e_{1},e_{2},e_{3},e_{4}\}$ be the standard basis of elementary vectors of the the finite vector space $\bb{F}_{3}^{4}$. Consider the natural action of $\rm{S}_{4}$ on the finite vector space $\bb{F}_{3}^{4}$. We note that $\A_{4}$ fixes $f:=e_{1}+e_{2}+e_{3}+e_{4}$. Now the quotient space $V=\bb{F}_{3}^{4}/\langle f \rangle$ is an $\A_{4}$-space. For $i=1,2,3$, let $f_{i}:=e_{i}+ \langle f \rangle$ and let $f_{4}:=-(f_{1}+f_{2}+f_{3})$, then the action of $\A_{4}$ on $V$ satisfies $\sigma\cdot(a_{1}f_{1}+a_{2}f_{2}+a_{3}f_{3})=\sum a_{i}f_{\sigma(i)}$, where $\sigma \in \A_{4}$ and $a_{1},a_{2},a_{3} \in \bb{F}_{3}$. We now consider the group $G= V {:} \A_{4}$. Let $C=\langle (1,2)(3,4)\rangle<\A_{4}$ , $\D$ be the Sylow-2-subgroup of $\A_{4}$, and $W:=\langle f_{1},\ f_{2}\rangle$. We now consider the subgroup $H:=W {:} C$ and the transitive action of $G$ on the coset space $\Ome:=[G \ :\ H]$. As all the conjugates of $(1,2)(3,4)$ lie in $\D$, we see that all conjugates of elements of $H$ must lie in the subgroup $K:=V {:} \D$, that is, $\cup_{g\in G} H^{g} \subset K$.

Now any $\pi\in \D$ is $\A_{4}$-conjugate to an element of $C$, so every element of $K$
is of the form $(af_{1}+bf_{2}+cf_{3},\pi)$, where $\pi$ is either $()$ or $(1,2)(3,4)$ (here $()$ denotes the identity permutation).
Recalling that $f_{4}=-f_{1}-f_{2}-f_{3}$, we now observe that
\begin{align*}
\left(\frac{(b-a)}{2} f_{3},\ (1,3)(2,4)\right)\left(af_{1}+bf_{2}+cf_{3},\ (1,2)(3,4)\right)\left(-\frac{(b-a)}{2}f_{1},\ (1,3)(2,4)\right)\in H.
\end{align*}
Given any $v \in V$ and $\pi \in \A_{4}$, we have $$(v,\pi)(af_{1}+bf_{2}+cf_{3},\ ())(-\pi^{-1}(v),\pi^{-1})=(af_{\pi(1)}+bf_{\pi(2)}+cf_{\pi(3)},\ ()).$$ Let $a_{\pi},b_{\pi}, c_{\pi}$ be such that  $af_{\pi(1)}+bf_{\pi(2)}+cf_{\pi(3)}=a_{\pi}f_{1}+b_{\pi}f_{2}+c_{\pi}f_{3}$. Straightforward computation shows that $U:=\{c_{\pi}|\ \pi \in \A_{4}\}=\{\pm a,\ \pm b\,\ \pm c,\ \pm(a-b),\ \pm(b-c),\ \pm(a-c)\}$. It is not difficult to check that for $(a,b,c) \in \bb{F}_{3}$, we have $0 \in U$, that is, there is a $\sigma \in \A_{4}$ such that $c_{\sigma}=0$. Then we have
$$(0,\sigma)(af_{1}+bf_{2}+cf_{3},\ ())(0,\sigma^{-1})=(a_{\sigma}f_{1}+b_{\sigma}f_{2},\ ()) \in H.$$

We just proved that every element $K$ is conjugate to some element of $H$. So by Lemma \ref{Int-subgp}, $K$ is an intersecting subgroup. Since, $\cup_{g\in G} H^{g} \subset K$, we must have $\cup_{g\in G} H^{g} =K $, and thus $K$ has to be a maximum intersecting set containing the identity. Therefore , $$\rho(G, \Omega)=|K|/|H|=6 >\sqrt{|\Omega|}=\sqrt{18}.$$ (We want to mention that $G\cong \bb{Z}_{3}^{3}:\A_{4}$, $H\cong \bb{Z}_{3}^{2}:\bb{Z}_{2}$, and $K\cong \bb{Z}_{3}^{3}:\bb{Z}_{2}^{2}$.) 
This example satisfies the parameters in row 6 of Table~\ref{larget}.
}
\end{example}

\section{Proofs of Theorem~\ref{qp-EKR}.}

Let $G$ be a quasiprimitive permutation group on $\Ome$. Then from the O'Nan-Scott theorem for quasiprimitive groups (c.f \cite{Pra1993}), we see that either (a) $G$ has a subnormal regular subgroup, (b) $G$ is an almost simple group, or (c) $G$ is a quasiprimitve group of product type action.

If $G$ has a regular subgroup, by Lemma \ref{regekr}, $G$ satisfies the EKR property. So every quasiprimitive group that does not satisfy the EKR property is either almost simple or is of product action type. Now, in order to finish the proof of Theorem~\ref{qp-EKR}, we need to show that
sets \[\{\rho(G, \Ome)|\ \text{$G\leqslant \Sym(\Ome)$ is an almost simple primitive permutation group}\}\] and \[\{\rho(G, \Ome)|\ \text{$G\leqslant \Sym(\Ome)$ is primitive permutation of product action type}\}\] are unbounded.

We now construct primitive almost simple groups $G \leqslant \Sym(\Ome)$ with large $\rho(G,\Ome)$.

\begin{construction}\label{ASC}{\rm
Let $q=2^{f}\geqslant 4$ and let $\bb{F}$ be the field of order $q$. We consider the action of $T=\PGL(2,q)\cong \PSL(2,q)$ on pairs of projective points, that is, on $\Delta=\{\{P,S\}| \ P,S \in \PG(1,q)\ \& \ P\neq S\}$. This action is a transitive permutation action.  Let $S$ be the maximal parabolic subgroup \[S=\left\{\begin{pmatrix}
1 & a\\
0 & b
\end{pmatrix}\ |\ a\in \bb{F}\ \& \ b \in \bb{F}^{\times}\right\}.\]}
\end{construction} 
Consider $\alpha=\{[1:0], [0:1]\}$, we note that
$T_{\alpha}=\frac{|T|}{|\Delta|}=2(q-1)$. Let $\gamma$ be a generator of the multiplicative group $\bb{F}^{\times}$.
We now observe that $r=\begin{pmatrix}
0 & 1\\
1 & 0
\end{pmatrix}$ and $s=\begin{pmatrix}
1 &0 \\
0 & \gamma
\end{pmatrix}$ fix $\alpha$ and thus $T_{\alpha} \cong \D_{2(q-1)}$. Since $T_{\alpha}$ is a maximal subgroup, $T$ is a simple primitive permutation group on $\Delta$.

Let $a \in \bb{F}$, then since characteristic of $\bb{F}$ is $2$, we see that $\begin{pmatrix}
1 & a\\
0 & 1
\end{pmatrix}$ fixes the set $\{[a:1],[0:1]\}$. When $b\neq 1 \in \bb{F}$, the element $\begin{pmatrix}
1 & a\\
0 & b
\end{pmatrix}$ of $S$ fixes $\{[1:0], [a:b-1]\}$.
 Since every element of $S$ fixes a point of $\Delta$, Lemma~\ref{Int-subgp} implies that $S$ is an intersecting subgroup. Thus $\rho(T,\Delta) \geqslant \dfrac{|S|}{|T_\alpha|}=\dfrac{q}{2}$. Therefore the set \[\{\rho(G, \Ome)|\ \text{$G\leqslant \Sym(\Ome)$ is an almost simple primitive permutation group}\}\] is unbounded. 
\begin{remark}{\rm
In the group $T$ constructed above, every element of $T_{\alpha}$ is also conjugate to some element in $S$. Thus we see that $X$ is an intersection sets for the permutation action of $T$ on $\Delta$ if and only if it is an intersection sets for the permutation action of $T$ on $\Gamma:=[T:S]$. The  action of $T$ on $\Gamma$ is a $2$-transitive action. In \cite{MS2011}, it was shown that the permutation action of $T$ on $\Gamma$ satisfies the strict-EKR property. Thus $S$ is a maximum intersection set for the action of $T$ on $\Delta$ and moreover every maxmimal intersection set is a coset of some conjugates of $S$. Therefore in this case, $\rho(T,\Delta)=q$. We observe that $\rho(T,\Delta)< \dfrac{\sqrt{|\Delta|}}{\sqrt{2}}$ and that $\lim\limits_{q\to \infty} \dfrac{\rho(T,\Delta)}{\sqrt{|\Delta|}}= \dfrac{1}{\sqrt{2}}$.}
\end{remark}

We now turn our focus to primitive groups of product type action.
\begin{construction}\label{PAC}{\rm
 Let $T$, $\Delta$, $\alpha$, $T_{\alpha}$, and $S$ be as in construction~\ref{ASC}. Given a prime $\ell$, we consider \[\begin{cases}P= T {\wr} \bb{Z}_{\ell}= T^{\ell}{:}\bb{Z}_{\ell}\\
\Lambda= \prod\limits_{l \in \bb{Z}_{\ell}} \Delta \end{cases}.\]
Then $P$ is a quasiprimitve group on $\Pi$ of product action type. The stabilizer of $\beta=(\alpha,\alpha,\ldots, \alpha)$  is $P_{\beta}=T_{\alpha}^{\ell}:\bb{Z}_{\ell}$. Pick the subgroup $\mathfrak{S}=S^{\ell}{:}\bb{Z}_{\ell}$. }
\end{construction}

 Now, as $S$ is an intersecting subgroup in $T$, by Lemma \ref{paty}, $\mathfrak{S}$ is an intersecting subgroup. We now have $\dfrac{|\mathfrak{S}|}{|P_{\beta}|}=\left(\dfrac{q}{2}\right)^{\ell}$. Thus $\rho(P,\Lambda)\geqslant \dfrac{|\mathfrak{S}|}{|P_{\beta}|}=\left(\dfrac{q}{2}\right)^{\ell}$. Therefore the set \[\{\rho(G, \Ome)|\ \text{$G\leqslant \Sym(\Ome)$ is a primitive permutation of product action type}\}\] is unbounded. 
 
 This concludes the proof of Theorem~\ref{qp-EKR}. \qed
\section{Proof of Theorem~\ref{PSLEKR}}\label{psl}
 In this section, we investigate EKR properties of primitive permutation groups isomorphic to $G=\PSL(2,p)$, where $p$ is a prime bigger than equal to $5$. Any primitive action of $G$ is equivalent to an action of $G$ on the coset space $[G:H]$ of a maximal subgroup $H$.

The subgroup structure of $\PSL(2,p)$ is well known. The following result describing the subgroups may be found in many classical texts such as \cite{Dickson, Huppert}.

\begin{theorem}\label{sub}{\rm(\cite[Hauptsatz 8.27]{Huppert})} For a prime $p \geqslant 5$, the subgroups of $G=\PSL(2,p)$ are isomorphic to one of the following groups:
\begin{enumerate}[{\rm(i)}]
\item $\bb{Z}_{p}{:}\bb{Z}_{\ell}$, where $\ell$ divides $\frac{p-1}{2}$;
\item $\bb{Z}_{\ell}$, where $\ell$ divides $\frac{p-1}{2}$ or $\frac{p+1}{2}$ ;
\item $\D_{\ell}$, where $\ell$ divides $p-1$ or $p+1$;
\item subgroups of $\A_{5}$, when $p \equiv \pm 1 \pmod{10}$;
\item $\A_{4}$, provided $p \equiv \pm 1 \pmod{4}$;
\item $\rm{S}_{4}$, provided $p \equiv \pm 1 \pmod{8}$.
\end{enumerate}
\end{theorem}

\begin{corollary} \label{max} For a prime $p \geqslant 5$, the maximal subgroups of $G=\PSL(2,p)$ are:
\begin{enumerate}[{\rm(i)}]
\item $\bb{Z}_{p}{:}\bb{Z}_{\frac{p-1}{2}}$;
\item $\D_{p-1}$, for $p\geqslant 13$;
\item $\D_{p+1}$, for $p \neq 7$;
\item $\A_{5}$, for $p \equiv \pm 1 \pmod{10}$;
\item $\A_{4}$, for $p \equiv \pm 3 \pmod{8}$;
\item $\rm{S}_{4}$, for $p \equiv \pm 1 \pmod{8}$.
\end{enumerate}
\end{corollary}

It is also known (for eg. see \cite{Huppert}) that all cyclic subgroups of $\PSL(2,p)$ of the same order are conjugate. This fact along with Lemma \ref{Int-subgp} gives us the following result.
\begin{lemma}\label{pi}
Let $G=\PSL(2,p)$ be a transitive permutation group on $\Ome$, and $\o\in\Ome$.
Then  $S<G$ is an intersecting subgroup if and only if $\{|s|\mid s\in S\}\subseteq\{|h|\mid h\in G_\o\}$.
\end{lemma}

Theorem \ref{sub}, Corollary \ref{max} and the above lemma equip us to find all intersecting subgroups of all primitive groups which are isomorphic to $G$.

Let $H$ be one of the maximal subgroups described in Corollary \ref{max}. We will now determine all intersecting subgroups with respect to the action of $G$ on $[G:H]$.

{\bf Case 1:} Let $H= \bb{Z}_{p} {:} \bb{Z}_{p-1}$. In this case, the action of $G$ on $[G:H]$ is $2$-transitive. In \cite{LPSX2018}, it was shown that this action satisfies the strict-EKR property.

{\bf Case 2:} Let $H=\D_{p-1}$ and $p\geqslant 13$. In view of Lemma \ref{pi}, a subgroup $S$ is an intersecting subgroup if and only if for all $n \in \{|s|\ | \ s\in S\}$, either $n=2$ or $n\mid \dfrac{p-1}{2}$. From Theorem \ref{sub} and Lemma \ref{pi}, we can now conclude that
\begin{itemize}
\item all cyclic intersecting subgroups are subgroups of $\bb{Z}_{\frac{p-1}{2}}$,
\item $\A_{4}$ is an intersecting subgroup if and only if $6 \mid \dfrac{p-1}{2}$ (note that $6$ is the exponent of $\A_{4}$),
\item $\rm{S}_{4}$ is an intersecting subgroup if and only if $12 \mid \dfrac{p-1}{2}$;
\item and $\A_{5}$ is an intersecting subgroup if and only if $30 \mid \dfrac{p-1}{2}$.
\end{itemize}

From this we conclude that in this case, weak-EKR property is satisfied whenever $p\neq 31$. All maximum intersecting subgroups of $\PSL(2,31)$ are of order $60$ are isomorphic to $\A_{5}$. We note that although $\D_{30}$ is a maximal subgroup (in $\PSL(2,31)$) which is an intersecting set, it is not a maximum intersecting subgroup since $|\D_{30}|<|A_5|$. We see that strict-weak-EKR holds whenever $p\neq 13, 31, 61$. We also observe that $\A_{4}$ is a maximum intersecting subgroup of $\PSL(2,13)$, and that $\A_{5}$ is a maximum intersecting subgroup of $\PSL(2,61)$.

{\bf Case 3:} Let $H=\D_{p+1}$ and $p\neq 7$.
Applying Theorem \ref{sub} and Lemma \ref{pi}, we can now conclude that
\begin{itemize}
\item all cyclic intersecting subgroups are subgroups of $\bb{Z}_{\frac{p+1}{2}}$,
\item $\A_{4}$ is an intersecting subgroup if and only if $6 \mid \dfrac{p+1}{2}$,
\item $\rm{S}_{4}$ is an intersecting subgroup if and only if $12 \mid \dfrac{p+1}{2}$;
\item and $\A_{5}$ is an intersecting subgroup if and only if $30 \mid \dfrac{p+1}{2}$.
\end{itemize}
Thus in this case, weak-EKR property holds whenever $p \neq 29$; and all maximum intersection subgroups of $\PSL(2,29)$ are isomorphic to $\A_{5}$. We note that although $\D_{30}$ is a maximal subgroup (in $\PSL(2,29)$) which is an intersecting set, it is not a maximum intersecting subgroup since $|\D_{30}|<|A_5|$. We can also conclude that strict-weak-EKR property holds whenever $p \neq 11, 23, 29, 59$. $\A_{4}$ is a maximum intersecting subgroup of $\PSL(2,11)$, $\rm{S}_{4}$ is a maximum intersecting subgroup of $\PSL(2,23)$, and $\A_5$ is a maximum intersecting subgroup of $\PSL(2,59)$. 

In the case $p \equiv 3\pmod{4}$, consider the subgroup $R=\bb{Z}_{p}{:}\bb{Z}_{\frac{p-1}{2}}$. Given $r \in R$, we have $gcd(|r|,p+1)=1$, and hence it does not fix any element of $[G:\D_{p+1}]$. Since $|[G:\D_{p+1}]|=|p(p-1)/2|$, we can see that $R$ is a regular subgroup. Therefore by Lemma~\ref{regekr}, we see that in this scenario, $G$ has the EKR property.

{\bf Case 4:} Let $p\equiv\pm 1 \pmod {10}$ and $H=\A_{5}$.
By Lemma \ref{pi}, a subgroup $S$ is an intersecting subgroup if and only if
each of its non-identity element is of order $2$, $3$, or $5$.
Any dihedral/cyclic group with elements of order both $3$ and $5$ has an element of order $15$.
As $p \equiv \pm 1 \pmod{10}$, $\D_{10}$ is the only intersecting subgroup which is isomorphic to a dihedral group.
Now using Theorem~\ref{sub}, we conclude that $\A_{5}$ is the only maximum intersecting subgroup.
Therefore strict-weak-EKR property is true in this case.

{\bf Case 5:} Let $p\equiv\pm 3 \pmod {8}$ and $H=\A_{4}$.
By Lemma \ref{pi}, a subgroup $S$ is an intersecting subgroup if and only if
each of its non-identity element is an element whose order is either $2$, or $3$.
Consider a dihedral group of order $2n>6$, with $3\mid n$. Such a group has an element of order $6$. Thus by Lemma~\ref{pi} and Theorem~\ref{sub}, no dihedral subgroup of $G$ is an intersecting subgroups. Using similar arguments on other subgroups described in Theorem \ref{sub}, we can conclude that strict-weak-EKR property holds in this case.

{\bf Case 6:} Let $p\equiv\pm 1 \pmod {8}$ and $H=\rm{S}_{4}$.
By Lemma \ref{pi}, a subgroup $S$ is an intersecting subgroup if and only if
each of its non-identity element is an element whose order is either $2$, $3$, or $4$.
Any dihedral/cyclic group with elements of order both $3$ and $4$ has an element of order $12$.
Also any dihedral group whose order is $4k$ with $k>2$  has an element of $8$.
Using similar order arguments to the subgroups described in Theorem \ref{sub},
we can conclude that strict-weak-EKR property holds in this case. This concludes the proof of Theorem~\ref{PSLEKR}. \qed

We now conclude the section with some observations. Let $p\geqslant 5$ be a prime and $G=\PSL(2,p)$ be a primitive permutation group on $\Ome$, with $G_{\o} \cong \D_{p+1}$ (with $\o \in \Ome$). The group $\tilde{G}=\PGL(2,p)$ also acts primitively on $\Ome$ with $\tilde{G}_{\o}\cong \D_{2(p+1)}$. Now consider the following example.
\begin{example}\label{PGLexam}
\rm{
Let $q$ be the power of an odd prime. Let $t$ be an integer such that $2^{t} || (q-1)$.
Consider $\tilde{G}=\PGL(2,q)$ and its subgroup $\tilde{H}$ generated by $\left(\begin{matrix}
0 & 1\\
\beta^{q+1} & 0
\end{matrix} \right)$ and $\begin{pmatrix}
 0 & 1 \\
 -\beta^{(1+q)} & \beta+\beta^{q}
 \end{pmatrix}$, where $\beta \in \mathbb{F}_{q^{2}}^{\times}$ is an element of order $2^{t}(q+1)$. We note that $\tilde{H}$ is isomorphic to $\D_{2(q+1)}$.

Let $F$ be the field of order $q$, and $t$ be such that $2^{t}|| (q-1)$. Let $M<F^{\times}$ be the unique multiplicative subgroup of index $2^{t}$. Let $\delta$ denote a primitive $2^{t}$-th root of unity and set
$\tilde{M}:=\bigcup\limits_{i=0}^{2^{t-1}-1}\delta^{i}M$. Now the set
\[\tilde{C}=\left\{\begin{pmatrix}
1 & b \\
0 & a
\end{pmatrix}\ \vert\ a\in \tilde{M}\ \&\ b \in F \right\}\]
} is a transversal for the cosets of $H$. We also note that for any $c_{1},c_{2} \in \tilde{C}$, the ratio $c_{1}c_{2}^{-1}$ is not conjugate to an element of $H \setminus \{1\}$. Thus $\tilde{C}$ is a sharply transitive subset and thus by Lemma~\ref{regekr}, the action of $\tilde{G}$ on $[\tilde{G}:\tilde{H}]$ has the EKR property. We note that this is the primitive action of $\PGL(2,q)$ on  pairs of imaginary points (that is, points of $\PG(1,q^{2})$ that do not lie in $\PG(1,q)$).

When $q \equiv 3\pmod{4}$, we observe that $\tilde{C} \subset G=\PSL(2,q)$. By arguments similar to those above, we see that the primitive action of $G$ on $[G:H]$, where $H=\tilde{H}\cap G \cong \D_{q+1}$, also satisfies the EKR property.
\qed
\end{example}

We show in Example~\ref{PGLexam} that $\tilde{G}$ has the EKR property. In the same example, we also see that $G$ satisfies the EKR property in the case $q\equiv 3\pmod{4}$. From Theorem~\ref{PSLEKR}, we know that $G$ has the strict-weak-EKR property if and only if $p\neq 29$. All maximum intersecting subgroups of $\PSL(2,29)$ are isomorphic to $\A_{5}$. We wonder if for $p\neq 29$, the group $G$ has the strict-EKR property. This example leads us to the following general question:
\begin{problem}
{\rm
Let $G\leqslant \Sym(\Omega)$ be a transitive group.
\begin{enumerate}[(a)]
\item Let $X$ be the set of maximum intersecting sets in $G$. Does there exist a subgroup $H \leqslant G$ such that $H \in X$?
\item If $G$ has the weak-EKR property, does it also have the EKR property?
\end{enumerate}
}
\end{problem}
We do not study this problem in this paper, but believe that this would be an interesting direction for further research.

\section{Intersecting sets of Frobenius groups}\label{FR}
We conclude this article by giving an alternate proof of Lemma~\ref{regekr} (which is equivalent to Corollary 2.2 of \cite{AM2015}). As a byproduct of our proof, we get the following theorem that describes all the intersecting sets in a Frobenius group.\begin{theorem}\label{frobekr}
Let $G$ be a Frobenius group with $K \leqslant G$ as its Frobenius kernal and $H \leqslant G$ as its a Frobenius complement. Then the following are true:
\begin{enumerate}[{\rm(i)}]
\item $S \subset G$ is an intersecting set if and only if there exists a set $\{H_{c}|\ c \in K\}$ of pairwise disjoint subsets of $H$ (that is $H_{c}\subset H$ and $H_{c}\cap H_{d} =\emptyset$ for all $c, d \in K$ with $c\neq d$) such that $S=\bigcup\limits_{c \in K} H_{c}c$; and
\item $G$ satisfies the strict-EKR property if and only if $|H|=2$.
\end{enumerate}
\end{theorem}
Part~(i) of this theorem may be deduced from Theorem 3.6 of \cite{AM2015} and part~(ii) is equivalent to Theorem 3.7 of \cite{AM2015}. The proofs in \cite{AM2015} rely on techniques from algebraic graph theory and representation theory. We will prove this Theorem as an application of Lemma \ref{regekr}.

\subsection{Proof of Lemma \ref{regekr}.}\label{pre}
We are given a finite permutation group $G$ on $\Omega$ and $\o \in \Omega$, such that $G$ contains a sharply transitive subset $C$. Recall that a subset $C$ of a permutation group $G \leq \Sym(\Ome)$ is said to be {\it sharply transitive} if for any $(\alpha,\beta) \in \Ome \times \Ome$ there is exactly one $c \in C$ such that $\alpha^{c}=\beta$. 

Consider $c_{1}, c_{2} \in C$ and $\alpha \in \Ome$ such that $\alpha^{c_{1}c_{2}^{-1}}=\alpha$. So we have $\alpha^{c_{1}}=\alpha^{c_{2}}$. As $C$ is a sharply transitive set, we must have $c_{1}=c_{2}$. Thus for distinct $c_{1},c_{2} \in C$, the ratio $c_{1}c_{2}^{-1}$ is a fixed point free permutation.
Moreover for any $\o \in \Ome$, $C$ is a set of right coset representatives for $G_{\o}$ in $G$. We thus have the decomposition
\[G= \bigcup\limits_{c \in C}G_{\o}c.\]

Let $S$ be an intersection set in $G$ and define $S_{c}=S \cap G_{\o}c$. So $S$ is a disjoint union of $S_{c}$'s, that is,
\[S= \bigcup\limits_{c \in C}S_c.\]

Now we consider the translations of $S_{c}$'s into $G_{\o}$, that is, the sets $H_{c}:=S_{c}c^{-1}$ for $c\in C$.
We then have
\begin{equation}\label{eq1}
|S|=\sum\limits_{c \in C}|S_{c}|=\sum\limits_{c \in C}|H_{c}|.
\end{equation}
Let $c,d \in C$ such that $H_{c} \cap H_{d}\neq \emptyset$. 
Let  $h \in H_{c}\cap H_{d}\subset G_{\o}$. 
So we have an element $x \in S_{c}\subset S$ and an element $y \in S_{d} \subset S$ such that $xc^{-1}=h=yd^{-1}$. Therefore, $cd^{-1}=xh^{-1}hy^{-1}=xy^{-1}$. 
Since $x,y$ are elements of an intersecting set, $xy^{-1}$ fixes a point.
As $c,d$ are elements of a sharply transitive set, $cd^{-1}$ fixes a point if and only if $c=d$. 
Thus $\{H_{c}\ |\ c \in C\}$ is a set of pairwise disjoint subsets of $G_{\o}$ such that 
$S=\bigcup\limits_{c \in C}H_{c}c$. 
So part~(a) is true. 
Now by \eqref{eq1}, we have $|S| \leqslant |H|$, and thus $G$ has EKR property.
\qed

\subsection{Proof of Theorem \ref{frobekr}.}
We may assume that $G$ is a Frobenius group acting on $K$. In this case, $H=G_{1}$, the stabilizer of the identity $1$ of $G$.
Let $\{H_{c}| c \in K\}$ be a set of pairwise disjoint subsets of $H$. Consider $S= \bigcup\limits_{c\in K} H_{c}c$. Consider elements $x,y \in \S$ such that $x \in H_{a}a$ and $y \in H_{b}b$, for some $a,b \in K$. If $a=b$, then clearly $xy^{-1}\in H_{a}\subset H$, and thus $xy^{-1}$ fixes a point. Now we assume that $a\neq b$. Then for some $h_{1}\in H_{a}$, $h_{2} \in  H_{b}$, we have \[xy^{-1}=h_{1}ab^{-1}h_{2}^{-1}=h_{1}ab^{-1}h_{1}^{-1}h_{1}h_{2}^{-1}.\] Now $H_{a}$ and $H_{b}$ are disjoint and hence $h_{1}h_{2}^{-1} \neq 1$. Also as $K$ is a normal subgroup, $h_{1}ab^{-1}h_{1}^{-1}h \in K$. Thus $xy^{-1}=hk$, for some $h\in H \setminus \{1\}$ and $k \in K \setminus \{1\}$. Therefore $xy^{-1}\notin K$. Since $G$ is a Frobenius group, every element not in $K$ fixes a point, and thus $xy^{-1}$ fixes a point. We have proved that $S$ is an intersecting set. Now application of Lemma \ref{regekr} leads to part~(i).

We now prove part~(ii).
Let us first assume that $|H|=2$ and $H=\{1,h\}$, it follows from part~(i) that every maximum intersecting set of $G$ is of the form $\{k_{1},hk_{2}\}$, for some $k_{1},k_{2}\in K$. This is a coset of the subgroup $\{1,hk_{2}k_{1}^{-1}\}$, which is the point stabilizer as $hk_{2}k_{1}^{-1}\notin K$. Therefore $G$ has the strict-EKR property in this case.

Let us now assume that $|H|\geqslant 3$. Let us pick $c \in K\setminus \{1\}$ and $h \in H \setminus\{1\}$. The by part~(i), $S= H\setminus\{h\} \cup \{ch\}$ is an intersecting subset. As $G$ is a frobenius group, and $ch \notin K$, the element $ch$ fixes only one point. As $ch \notin H$,  it does not fix $1$. Since every other point of $S$ fixes $1$, the set $S$ cannot be a coset of a point stabilizer.
\qed
We end this section with an example of a non-Frobenius group, in which we find intersecting sets that are not cosets of subgroup. In order to do so we use the description of intersecting sets given in Lemma \ref{regekr}.
\begin{example}
{\rm
Let $G=\SL(2,3)=\Q_8{:}\ZZ_3$, and $H=\ZZ_3$.
Then $G$ is not a Frobenius group.
Let $\Ome=[G:H]$.
Then $\Q_8$ acting on $\Ome$ is a regular subgroup, and hence $G$ has the EKR-property on $\Ome$ by Lemma~\ref{regekr}.

Let $H=\l h\r$ and $a\in K=\Q_8$ be of order 4.
Let $H=\{1\}\cup\{h\}\cup\{h^{-1}\}$, and
\[S=1\{1\}\cup a\{h\}\cup a^{-1}\{h^{-1}\}=\{1,ah,a^{-1}h^{-1}\}.\]
Then $S$ is not a coset of any subgroup of $G$, and so $G$ does not have the strict EKR-property.
}
\qed
\end{example}
\section{Proof of Theorem~\ref{ps}.}\label{pp}
In this section, we show that all permutation groups of prime power degree satisfy the EKR property. 

At first, we consider transitive permutation groups which are $p$-groups. We show that such groups satisfy EKR property. We recall from Section \ref{prel} that a sharply transitive set of a permutation group $K \leq \mathrm{Sym}(\Delta)$ is a set $C$ such that for every $\alpha,\beta \in \Delta$, there is a unique $c \in C$ such that $\alpha^{c}=\beta$. Lemma \ref{regekr} states that the existence of such a set implies EKR property. The following lemma shows the existence of sharply transitive sets in all transitive permutation $p$-groups 
\begin{lemma}
Let $p$ be a prime, then every transitive permutation group $p$-group has a sharply transitive set.
\end{lemma}      
\begin{proof}
Let us consider a fixed prime $p$, the the Sylow-$p$-subgroup of $\mathrm{S}_{p}$ is a regular subgroup and thus $p$-subgroups of $\mathrm{S}_{p}$ have a sharply transitive subset. Let $k$ be any integer, and assume that for $n<k$, every transitive $p$-subgroup of $\mathrm{S}_{p^{n}}$ has a sharply transitive set. 
Let $P\leq \mathrm{S}_{p^{k}}$ be a transitive $p$-subgroup. Let $Z$ be the centre of $P$. As the centre of a $p$-group is non-trivial, $Z$ is a non-trivial normal subgroup  of $P$. As stabilizer of $P$ is core-free, $Z$ acts semi-regularly. Thus the action of $Z$ has $p^{s}:=|p^{k}|/|Z|$ orbits. As $Z$ is non-trivial, we have $s<k$.

If $s=0$, then $Z$ is a regular subgroup, and thus $P$ has a sharply transitive set. 

If $s>1$, consider the set $\Lambda$ of $Z$-orbits. This forms a block system for the action of $P$, and thus $P$ acts transitively on $\Lambda$. Let $N$ be the kernal of this action, clearly $Z\leq N$. 

 Let $R:=P/N \leq \mathrm{Sym}(\Lambda)\cong \mathrm{S}_{p^{s}}$ be the permutation induced by the action of $P$. By our assumption, $R$ has a sharply transitive set, say $C_{0}=\{g_{1}N,\ldots g_{p^{s}}N\}$. Now $C=\bigcup\limits_{i=1}^{p^{s}} g_{i}Z \subset P$ is a sharply transitive subset of $P$.     
\end{proof}
Let $G \leq \mathrm{S}_{p^{k}}$ be any transitive group, and let $H$ be a stabilizer. Pick a Sylow-$p$ subgroup $O$ of $H$, and pick a Sylow-$p$ subgroup $P$ of $G$ such that $O <P$. We have \[p^{k}=[G:H]= \dfrac{[G:P]}{[H:O]}[P:O].\] As both $[G:P]$ and $[H:O]$ are co-prime with $p$, we have $[P:O]=p^{k}$. As $H\cap P= O$, we conclude that $P$ acts transitively. Now from the above Lemma, $P$ contains a sharply transitive subset, and thus so does $G$. This shows that $G$ satisfies the EKR property. We have now proved Theorem~\ref{ps}.
 \qed

\bibliographystyle{plain}

\end{document}